\documentclass[12pt,reqno]{amsart}

\usepackage[
  margin=1in,
  includefoot,
  footskip=30pt,
]{geometry}

\usepackage{amsfonts}
\usepackage{mathrsfs}
\usepackage{amssymb}

\usepackage{amsthm}

\usepackage{enumitem}

\usepackage{tikz-cd}

\usepackage{hyperref}
\usepackage{cleveref}

\theoremstyle{definition}
\newtheorem*{definition}{Definition}

\theoremstyle{plain}
\newtheorem{theorem}{Theorem}
\newtheorem{proposition}[theorem]{Proposition}
\newtheorem{lemma}[theorem]{Lemma}

\newtheorem*{corollary}{Corollary}
\newtheorem*{theorem*}{Theorem}

\newtheorem{mainlemma}{Lemma}

\newtheorem{maintheorem}[mainlemma]{Theorem}

\def\op{\operatorname}
\def\E{\mathscr{E}}
\def\F{\mathscr{F}}
\def\G{\mathscr{G}}
\def\I{\mathscr{I}}
\def\K{\mathscr{K}}
\def\C{\mathfrak{C}}
\def\D{\mathfrak{D}}
\def\L{\mathscr{L}}
\def\O{\mathscr{O}}

\def\P{\mathbf{P}}

\def\Spec{\op{Spec}}
\def\Fitt{\mathscr{F}itt}

\def\Ext{\op{Ext}}

\def\sExt{\mathscr{E}\kern -.5pt xt}
\newcommand*{\sHom}{\mathscr{H}\kern -.5pt om}

\def\coker{\op{coker}}

\def\im{\op{im}}
\def\rank{\op{rank}}
\def\grade{\op{grade}}
\def\depth{\op{depth}}

\begin{document}

\title{Graded basic elements}
\author{Mengyuan Zhang}
\address{Department of Mathematics, University of California, 
Berkeley, CA 94720}
\email{myzhang@berkeley.edu}

\begin{abstract}
We prove the existence theorem for basic elements in the quasi-projective case, extending results of Eisenbud-Evans and Bruns from the affine case.
We give several geometric applications.
For example, we show that every local complete intersection of pure codimension two in a smooth projective variety of dimension $d$ over an infinite field is the degeneracy locus of $d-1$ sections of a rank $d$ bundle.
\end{abstract}

\maketitle 

\section*{Introduction}
The notion of basic elements originated in the work of Forster \cite{Forster} and Swan \cite{Swan}.
A basic element of a module is an element that is part of a minimal system of generators of the module at the localization by every prime.
The theory surrounding basic elements has always been very technical, but its applications have shown themselves to be very useful.
For example, the Forster-Swan theorem provides an upper bound on the number of elements needed to generate a finite module over a ring.
Furthermore, basic elements are closely related to unimodular elements, which were central to Bass' work \cite{Bass} on cancellation and stable range theorems in $K$-theory.
In \cite{EE}, Eisenbud and Evans generalized and simplified the surrounding ideas, and proved an existence theorem of basic elements in submodules that contain ``enough" minimal generators at each point.
Bruns \cite{Bruns} proved an analogous existence theorem of elements that are basic at primes up to a given depth.

\bigskip

The first goal of this article is to extend the above existence theorems to the graded setting, more precisely, to coherent sheaves on a quasi-projective scheme over a noetherian ring containing an infinite field.
Instead of stating the technical results here, we produce below a consequence of the main theorems, which generalizes a lemma of Serre (see \cite[p148]{Mumford}) on the existence of a non-vanishing section of a globally generated bundle whose rank is greater than the dimension of the ambient variety.

\setcounter{theorem}{-1}
\begin{theorem}\label{Intro}
Let $X$ be a scheme of dimension $d$ that is quasi-projective over a noetherian ring containing an infinite field.
Suppose $\F$ is a coherent sheaf on $X$ generated by global sections and $\mu(\F_p) > d$ for all $p\in X$.
Then there is a section $s$ of $\F$ such that if $\F'$ is the subsheaf generated by $s$ then $\mu((\F/\F')_p) = \mu(\F_p)-1$ for all $p\in X$. 
Here $\mu(\F_p)$ denotes the minimal number of generators of $\F_p$ over the local ring $\O_p$.
\end{theorem}

The second goal of this article is to provide several geometric applications of the theory of graded basic elements.
We define the Cayley-Bacharach index of a set of points $Z$ in $\P^2_k$, which is a geometric invariant, and bound it above and below by the highest and lowest second Betti numbers of the homogeneous ideal of $Z$, which are algebraic invariants. 
We also prove that every local complete intersection of pure codimension two in a smooth projective variety of dimension $d$ is the degeneracy locus of $d-1$ sections of a rank $d$ bundle.

\setcounter{section}{-1}
\section{Setup}
Throughout, let $A$ denote a noetherian ring containing an infinite field $k$, and let $X$ be a quasi-projective scheme over $A$ with a very ample line bundle $\O(1)$.
All sheaves in consideration are assumed to be coherent on $X$. 

\bigskip

For a sheaf $\F$, we write $\F(l)$ for $\F\otimes \O(1)^{\otimes l}$ and define $H^i_*(\F)$ to be $\bigoplus_{l \in\mathbb{Z}} H^i(X,\F(l))$, which is a graded module over the section ring $H^0_*(\O)$.
We call sections of $\F(l)$ \emph{sections of $\F$ in degree $l$}, or \emph{twisted sections of $\F$}. 
If $s_1,\dots, s_u$ are sections of $\F$ in degrees $a_1,\dots, a_u$, then we use $(s_1,\dots, s_u)$ to denote the image subsheaf of the morphism
\[
\varphi:\bigoplus_{i = u}^l \O(-a_i)\xrightarrow{s_1,\dots,s_u} \F.
\]
For a point $p\in X$, let $\F_p$ denote the stalk of $\F$ at $p$, and let $\F_{(p)}$ denote the fiber of $\F$ at $p$. 
We write $\mu(\F_p)$ for the minimal number of generators of $\F_p$ over the local ring $\O_p$. 
Nakayama's lemma implies that $\mu(\F_p) = \dim_{k(p)}\F_{(p)}$.
For any $p\in X$, there exists a linear form $L\in H^0(\O(1))$ which does not vanish at $p$ since $\O(1)$ is very ample. 
If $\varphi$ is as above, then the image of $\varphi_p$ is generated by $\{(s_1/L^{a_1})_p,\dots, (s_u/L^{a_u})_p\}$ in $\F_p$. 
If we chose a different dehomogenization, then these elements would be scaled by units in $\O_p$.
The image of $\varphi_{(p)}$ is generated by $\{(s_1/L^{a_1})_{(p)},\dots, (s_u/L^{a_u})_{(p)}\}$ in the fiber $\F_{(p)}$ and does not depend on the choice of $L$.
The $k(p)$-dimension of $\im \varphi_{(p)}$ is equal to $\mu(\F_p)-\mu((\coker \varphi)_p)$ in view of the right exact sequence 
\[
\bigoplus_{i = u}^l \O(-a_i)_{(p)} \xrightarrow{\varphi_{(p)}} \F_{(p)} \to (\F/(s_1,\dots, s_u))_{(p)} \to 0.
\]

\section{The Finite Shrinking Lemma}
The central definition of this article is the following.
\begin{definition}
A subsheaf $\F'$ of a sheaf $\F$ is said to be \emph{$w$-basic in $\F$ at $p\in X$} if 
\[
\mu((\F/\F')_p)\le \mu(\F_p)-w.
\]
Twisted sections $s_1,\dots, s_u$ of $\F$ are said to be \emph{basic in $\F$ at $p\in X$} if $(s_1,\dots, s_u)$ is $u$-basic in $\F$ at $p$.
\end{definition}

By Nakayama's lemma, a subsheaf $\F'$ is $w$-basic in $\F$ at $p$ iff $\F'_p$ contains $w$ elements that are a part of a minimal system of generators of $\F_p$.

\bigskip

Our main technical contribution is the following lemma, which says that we may, after a unipotent change of coordinates, drop one section of the lowest degree while maintaining basicness to the maximal possible extent at finitely many points.

\begin{mainlemma}[Finite shrinking lemma]\label[lemma]{FiniteBasic}
Let $\{p_1,\dots, p_v\}$ be a finite set of points in $X$.
Let $s_1,\dots, s_u$ be sections of a sheaf $\F$ in degrees $a_1 \le \dots \le a_u$ where
\[
(s_1,\dots, s_u)\text{ is }w_i\text{-basic in }\F\text{ at }p_i\quad \forall 1\le i\le v.
\]
Then there are $r_j \in H^0(\O(a_j-a_1))$ for each $2\le j\le u$ such that 
\[
(s_2+r_2\cdot s_1,\dots, s_u+r_u\cdot s_1)\text{ is }\min(u-1,w_i)\text{-basic in }\F\text{ at }p_i\quad \forall1\le i \le v.
\]
\end{mainlemma}

\Cref{FiniteBasic} is the graded version of \cite[Lemma 3]{EE}, which is a crucial step in both \cite[Theorem A]{EE} and \cite[Theorem 1]{Bruns}.
We explain below the new difficulty in the graded setting.

\bigskip

If $R$ is a ring and $M$ is an $R$-module, then any element $s_j$ of $M$ can be modified by a multiple of a given element $s_1$ to yield another element $s_j' = s_j + r_js_1$.  
However, in the setting of \Cref{FiniteBasic}, the elements $s_1$ and $s_j$ of the module $H^0_*(\F)$ are homogeneous of degrees $a_1$ and $a_j$, and the form $r_j$ must now be homogeneous of degree $a_j-a_1$. 
Since $H^0_*(\O)$ may not contain homogeneous forms of negative degrees (e.g. when $X$ is integral), we can modify higher degree sections by lower degree ones but not vice versa in general.
The original proof of \cite[Lemma 3]{EE} needed a form $r_j$ that vanishes on $p_1,\dots, p_{v-1}$ but not on $p_v$, assuming $p_v$ is minimal. 
Although such a homogeneous form exists in high enough degrees in $H^0_*(X)$, there is no guarantee for it to exist in degree $a_j-a_1$. 

\bigskip

We use a different strategy centered around the following lemma from linear algebra.

\begin{lemma}\label{LinearAlg}
Let $v_1,\dots,v_u$ be vectors in a vector space $V$ over a field $K$.
For any fixed $j$ where $2\le j\le u$, there is at most one $\lambda \in K$ where
\[
\dim \op{span}\{v_2,\dots, v_j+\lambda\cdot v_1,\dots, v_u\} <\dim \op{span}\{v_2,\dots, v_u\}.
\] 
\end{lemma}
\begin{proof}
If $v_j \in \op{span}\{v_2,\dots,\hat{v_j},\dots, v_u\}$, then clearly
\[
\dim \op{span}\{v_2,\dots, v_j+\lambda\cdot v_1,\dots, v_u\} \ge \dim \op{span}\{v_2,\dots, \hat{v_j},\dots v_u\} = \dim \op{span}\{v_2,\dots, v_u\}.
\] 
Suppose $v_j \not\in \op{span}\{v_2,\dots,\hat{v_j},\dots, v_u\}$. 
We may quotient $V$ by $\op{span}\{v_2,\dots, \hat{v_j},\dots, v_u\}$ and denote the images of $v_1$ and $v_j$ by $\overline{v_1}$ and $\overline{v_j}$.
If the inequality in the statement holds, then we must have $\overline{v_j}+\lambda\cdot \overline{v_1} = 0$.
It follows that $\overline{v_1}$ is a nonzero multiple of $\overline{v_j}$ and $\lambda$ is uniquely determined.
\end{proof}

\begin{lemma}\label{Nonvanishing}
Let $p_1,\dots, p_v\in X$ be finitely many points.
There is a linear form $L\in H^0(\O(1))$ that does not vanish on $p_i$ for any $1\le i \le v$.
\end{lemma}
\begin{proof}
Sections that vanish on $p_i$ form a proper $A$-submodule of $H^0(\O(1))$ for any $i$.
The conclusion follows from \Cref{Avoidance}.
\end{proof}

\begin{lemma}\label{Avoidance}
A nonzero $A$-module is not the union of finitely many proper submodules.
\end{lemma}
\begin{proof}
We claim that a nonzero vector space over an infinite field is not the union of finitely many proper subspaces.
Suppose $V$ is the union of proper subspaces $V_1,\dots, V_n$.
We may enlarge $V_i$ to be codimension one subspaces.
We may quotient $V$ by the intersection of $V_1,\dots, V_n$ and assume the intersection is zero.
In this case, the space $V$ embeds into the direct sum of the quotients $V/V_i$, and therefore $V$ is finite dimensional.
This argument shows that a counter-example in the infinite dimensional case gives a counter-example in the finite dimensional case.
But the statement is true for finite dimensional vector spaces.
\end{proof}

The above statement is not true in general without the assumption that $A$ contains an infinite field.
For example any finite dimensional vector space over a finite field is a finite union of proper subspaces.

\begin{proof}[Proof of \Cref{FiniteBasic}]
Let $L \in H^0(\O(1))$ be a form that does not vanish at $p_i$ for any $i$ as in  \Cref{Nonvanishing}.
We proceed by induction on $v$.

Suppose $v = 1$, and $(s_1,\dots, s_u)$ is $w_1$-basic in $\F$ at $p_1$.
If $u = w_1$, then any choice of $r_j$'s would work.
We may suppose $u>w_1$. 
If $(s_2,\dots, s_u)$ is already $w_1$-basic in $\F$ at $p_1$, then we may choose $r_j = 0$ for all $j$. 
If not, the elements $\{(s_2/L^{a_2})_{(p_1)}, . . . , (s_u/L^{a_u} )_{(p1)}\}$ form a linearly dependent set of vectors in the vector space $\F_{(p_1)}$, and there exists some $2\le l \le u$ such that $(s_l/L^{a_l})_{(p_1)}$ is in the span of $\{(s_{l+1}/L^{a_{l+1}})_{(p_1)},\dots, (s_u/L^{a_u})_{(p_1)}\}$.
In this case, we may take $r_j = 0$ for every $j\ne l$, and choose $r_l = L^{a_j-a_1}$.
It follows that $(s_2,\dots, s_l+r_ls_1,\dots, s_u)$ is $w_1$-basic in $\F$ at $p_1$.

Now we assume $v>1$.
If $w_i = u$ for some $i$, then any choice of $r_j$'s would satisfy the requirement at $p_i$. 
Thus we may consider the same problem but at fewer points, and induction on the number of points $v$ takes care of this case.
We may assume that $u>w_i$ for all $i$.
By the induction hypothesis, there are forms $r_j'\in H^0(\O(a_j-a_1))$ for $2\le j\le u$ such that $s_j':=s_j+r_j's_1$ generate a subsheaf that is $w_i$-basic in $\F$ at $p_1,\dots, p_{v-1}$. 
If $(s_2',\dots, s_u')$ is already $w_v$-basic in $\F$ at $p_v$, then we are done. 
Otherwise, the subsheaf $(s_2',\dots, s_u')$ is $(w_v-1)$-basic in $\F$ at $p_v$.
By the same reason above, there exists $2\le l \le u$ such that $(s_l'/L^{a_l})_{(p_v)}$ is in the span of $\{(s_{l+1}'/L^{a_l+1})_{(p_v)},\dots, (s_{u}'/L^{a_u+1})_{(p_v)}\}$.
We choose $r_l'' := \lambda \cdot L^{a_l-a_1}$ for a nonzero $\lambda \in k$ yet to be determined.
Since $r_l''$ does not vanish on $p_v$, the subsheaf $(s_2',\dots, s_l'+r_l''s_1,\dots, s_u')$ is $w_v$-basic at $p_v$ for any nonzero $\lambda\in k$. 
At each point $p_1,\dots, p_{v-1}$, there is at most one $\lambda\in k$ where $(s_2',\dots, s_l'+r_l''s_1,\dots, s_u')$ fails to be $w_i$-basic at $p_i$ by \Cref{LinearAlg}. 
Since $k$ is an infinite field, we may choose a nonzero $\lambda \in k$ such that $(s_2',\dots, s_l'+r_l''s_1,\dots, s_u')$ remains $w_i$-basic at $p_i$ for all $1\le i\le v-1$.
\end{proof}

We do not know if \Cref{FiniteBasic} remains true without the assumption that $A$ contains an infinite field.

\section{Existence Theorems and Reduction Theorems}

With \Cref{FiniteBasic} in hand, we can now prove two existence theorems of graded basic elements (\Cref{BasicCodim} and \Cref{BasicDepth}) following the treatments of their affine versions in \cite{EE} and \cite{Bruns}.
Since this effort is largely that of a translation, we defer the proofs to the appendix.
Instead, we explore some reduction theorems as consequences of the existence theorems.

\bigskip

In the following, for any natural number $m\in\mathbb{N}$ we define 
\[
\C_m := \{p\in X\mid \dim \O_p \le m\}\text{ and }
\D_m := \{p\in X\mid \depth \O_p \le m\}.
\]

\begin{maintheorem}[Existence theorem, codimension version]\label{BasicCodim}
Let $s_1,\dots, s_u$ be sections of a sheaf $\F$ in degrees $a_1\le \dots \le a_u$. 
Suppose for some $t\ge 1$
\[
(s_1,\dots, s_u) \text{ is }\min(u,m+t-\dim \O_p) \text{-basic in }\F \quad \forall p\in \C_m.
\]
\begin{enumerate}[leftmargin = *,parsep = 0pt]
\item If $u \le t$, then $s_1,\dots, s_u$ are basic in $\F$ at all $p\in \C_m$ trivially. 
\item If $u > t$, then there are $t$ sections $s'_{u-t+1},\dots, s_u'$ of $\F$ in the highest degrees $a_{u-t+1},\dots, a_u$ that are basic in $\F$ at all $p\in \C_m$.
Furthermore, for each $u-t+1\le i \le u$, the section $s_i'$ can be chosen of the form $s_i+r_{i-1} s_{i-1}+\dots + r_1 s_1$ for some $r_j\in H^0(\O(a_i-a_j))$.
\end{enumerate}
\end{maintheorem}

Phrased in English, if the condition of the theorem is satisfied then after a unipotent change of coordinates we can choose $\min(u,t)$ sections of highest degrees that are basic in $\F$ at all points of codimension up to $m$ in $X$.

\begin{maintheorem}[Existence theorem, depth version]\label{BasicDepth}
Let $\F$ be a sheaf where 
\begin{align*}
\mu(\F,q)\le \mu(\F,p)\quad \forall p,q\in X \text{ such that }\depth \O_q < \depth \O_p \le m \tag{$\dagger$}\label{Cond}.
\end{align*} 
Let $s_1,\dots, s_u$ be sections of $\F$ in degrees $a_1\le \dots \le a_u$. 
Suppose for some $t\ge 1$
\[
(s_1,\dots, s_u)\text{ is }\min(u,m+t-\depth \O_p)\text{-basic in }\F \quad \forall p\in \D_m.
\]
\begin{enumerate}[leftmargin=*,parsep = 0pt]
\item If $u \le t$, then $s_1,\dots, s_u$ are basic in $\F$ at all $p\in \D_m$ trivially. 
\item If $u > t$, then there are $t$ sections $s'_{u-t+1},\dots,s_u'$ of $\F$ in the highest degrees $a_{u-t+1},\dots, a_u$ that are basic in $\F$ at all $p\in \D_m$.
Furthermore, for each $u-t+1\le i \le u$, the section $s_i'$ can be chosen of the form $s_i+r_{i-1} s_{i-1}+\dots + r_1 s_1$ for some $r_j\in H^0(\O(a_i-a_j))$.
\end{enumerate}
\end{maintheorem}

Note that if $\F$ is locally-free at all points $p\in \D_m$, then $\F$ satisfies the condition (\ref{Cond}).

\bigskip

A new feature in the theory of graded basic elements is the choice of degrees.
Since we can only modify sections of higher degrees by those of lower degrees, graded basic elements produced this way tend to have high degrees. 
Therefore, graded basic elements of low degrees are particularly interesting.
In fact, there is always a minimal choice of degrees.
See \cite{Zhang} for an exploration of this idea which leads to the definition of a minimal sheaf.

\bigskip

The theorem stated in the introduction follows immediately from \Cref{BasicCodim}.

\newtheorem*{intro}{\Cref{Intro}}
\begin{intro}
If $\dim X = d$ and $\F$ is a globally generated sheaf where $\mu(\F_p) > d$ for all $p\in X$,
then there is a section $s$ of $\F$ such that $\mu((\F/(s))_p) = \mu(\F_p)-1$ for all $p\in X$. 
\end{intro}
\begin{proof}
Let $ \O^u \to  \F$ be a surjective map corresponding to sections $s_1,\dots, s_u$. 
The assumption on $\F$ implies that $(s_1,\dots, s_u)$ is $\min(u, n+1-\dim \O_p)$-basic in $\F$ at all $p\in X$.
An application of \Cref{BasicCodim} with $t = 1$ yields the claim. 
\end{proof}

\begin{corollary}[Serre {\cite[p148]{Mumford}}]
If $X$ is an algebraic variety of dimension $d$ and $\E$ is a globally generated bundle on $X$ of rank $r >d$, then $\E$ has a section $s$ that is nowhere vanishing.
\end{corollary}

In particular, every vector bundle $\E$ can be ``reduced" to a bundle of rank at most $\dim X$ by modding out nowhere vanishing sections. 
We generalize this reduction theorem when $\E$ is only locally-free up to a given codimension.
 
\begin{definition}
We say a sheaf $\F$ is $(F_m)$ if $\F_p$ is free over $\O_p$ for all $p\in \C_m$.  
\end{definition}

Recall that a sheaf $\F$ has \emph{well-defined rank} $r$ if $\F_p$ is free over $\O_p$ of rank $r$ for all associated points $p$ of $X$.

\begin{theorem}\label{Reduction}
Suppose all associated points of $X$ have codimension $\le m$.
Let $\F$ be an $(F_m)$ sheaf of well-defined rank $r>m$. 
For any surjective map $\bigoplus_{i = 1}^u\O(-a_i) \xrightarrow{\varphi} \F$, there exists a direct summand $\L$ of $\bigoplus_{i = 1}^u\O(-a_i)$ of rank $(r-m)$ that injects into $\F$ via $\varphi$, such that $\F/\L$ is $(F_m)$ of well-defined rank $m$.
Furthermore, we can choose $\L$ to be a summand involving the largest $(r-m)$ integers $a_i$.
\end{theorem}
\begin{proof}
Let $(s_1,\dots, s_u)$ be the sections of $\F$ in degrees $a_1\le \dots \le a_u$ corresponding to $\varphi$. 
Since $(s_1,\dots, s_u)$ is $\min(u,r)$-basic in $\F$ at all $p\in \C_m$, it is $\min(u,m+(r-m)-\dim \O_p)$-basic in $\F$ at all $p\in \C_m$. 
By \Cref{BasicCodim}, there is a rank $(r-m)$ summand $\L$ of $\bigoplus_{i = 1}^u\O(-a_i)$  mapping to $\F$ giving $(r-m)$ sections that are part of a basis of $\F_p$ at each $p\in \C_m$. 
It follows that $\F/\varphi(\L)$ is locally free at all $p\in \C_m$. 
Since all associated points of $X$ are contained in $\C_m$, it follows that $\F/\varphi(\L)$ has well-defined rank $m$. 
We also conclude that $\varphi$ is injective since it is injective at all associated points of $X$ and $\L$ is locally-free.
\end{proof}

The condition on $X$ in the above theorem is mild. 
For example it is satisfied when $X$ is integral, or normal, or Cohen-Macaulay etc.

\bigskip

There is a parallel reduction theorem, due to Bruns in the affine case.

\begin{definition}
We say $\F$ is $(T_m)$ if $\depth \F_p \ge \min(m,\depth\O_p)$.
\end{definition}
If $X$ is Gorenstein in codimension $m$, then the condition $(T_m)$ is the same as being $m$-torsionless (see \cite{Bruns}).
If $X$ is Cohen-Macaulay, then the condition $(T_m)$ is the same as the Serre's condition $(S_m)$. 

\begin{theorem}[cf. {\cite[Theorem 2]{Bruns}}]\label{Reduction2}
Let $\F$ be a $(T_m)$ sheaf of well-defined rank $r>m$.
Suppose $\F_p$ has finite projective dimension over $\O_p$ for all $p\in \D_m$.
For any surjective map of the form $\bigoplus_{i = 1}^u\O(-a_i) \xrightarrow{\varphi} \F$, there exists a direct summand $\L$ of $\bigoplus_{i = 1}^u\O(-a_i)$  of rank $(r-m)$ that injects into $\F$ via $\varphi$, such that $\F/\L$ is $(T_m)$ of well-defined rank $m$.
Furthermore, we can choose $\L$ to be a summand of $\bigoplus_{i = 1}^u\O(-a_i)$ involving the largest $(r-m)$ integers $a_i$.
\end{theorem}
\begin{proof}
For $p\in \D_m$, the module $\F_p$ has finite projective dimension over $\O_p$ and $\depth \F_p = \depth \O_p$ by the assumption of $(T_m)$.  
It follows from the Auslander-Buchsbaum formula that $\F_p$ is in fact free over $\O_p$.

Let $s_1,\dots, s_u$ be twisted sections of $\F$ corresponding to $\varphi$. 
Since $(s_1,\dots, s_u)$ is $\min(u,m+(r-m)-\depth \O_p)$-basic in $\F$ at all $p\in \D_m$, it follows from \Cref{BasicDepth} that we may find a rank $(r-m)$ summand $\L$ corresponding to sections of the highest $(r-m)$ degrees that form part of a basis of $\F_p$ at each $p\in \D_m$. 
It follows that $\F/\varphi(\L)$ is locally free at all $p\in \D_m$. 
Since $\varphi:\L \to \F$ is injective at all $p\in \D_0$, i.e. all associated points of $X$, we see that $\varphi$ is injective and $\F/\L$ has well-defined rank $m$.

If $p\in X$ is such that $\depth \O_p>m$, then the depth lemma applied to the exact sequence
\[
0 \to \L_p \to \F_p \to (\F/\L)_p \to 0
\] 
implies that $\depth (\F/\L)_p \ge \min(\depth \F_p,\depth \L_p-1)\ge m$.
\end{proof}

In the following, we define the reduction of a sheaf and prove a factorization theorem.

\begin{definition}
A \emph{reduction} of a sheaf $\F$ is a surjective map $f:\F \to \F'$, where $\ker f$ is of the form $\bigoplus_{i = 1}^u \O(-a_i)$ for some integers $a_i$.
We say a reduction $f:\F \to \F''$ \emph{factors through} another reduction $g:\F \to \F'$ if there is a reduction $h:\F'\to \F''$ such that the following diagram commutes
\[
\begin{tikzcd}
\F \arrow[r,"g"] \arrow[rd,swap,"f"] & \F' \arrow[d,"h"]\\
 & \F''.
\end{tikzcd}
\] 
\end{definition}

\begin{maintheorem}[Factorization theorem]\label{Factor}
Let $f:\F \to \F''$ be a reduction of a sheaf $\F$.
Suppose $n\ge m$ and $v\ge u$ are four natural numbers where the following hold
\begin{enumerate}[leftmargin = *, parsep = 0pt]
\item $\F''$ is $(T_m)$ of well-defined rank $u$ and $u+\rank \ker f \ge v$,
\item $\F$ is $(T_n)$ and $\F_p$ has finite projective dimension over $\O_p$ for all $p\in \D_n$, 
\item $\mu(\F''_p) \le v-n+\depth \O_p$ for all $p\in \D_n-\D_m$.
\end{enumerate}  
Then $f$ factors through a reduction $g:\F \to \F'$ where $\F'$ is $(T_n)$ of well-defined rank $v$.
Moreover, the factor map $h:\F' \to \F''$ can be chosen such that $\ker h$ is a summand of $\ker f$ involving the smallest $(v-u)$ degrees.
\end{maintheorem}
\begin{proof}
Since both $\ker f$ and $\F''$ have well-defined rank, so does $\F$. 
Let $r$ denote the rank of $\F$ and let $s_1,\dots, s_{r-u}$ denote the twisted sections of $\F$ corresponding to $\ker f \to \F$.
Assumption (2) implies that $\F_p$ is free over $\O_p$ for all $p\in \D_n$ as in the proof of \Cref{Reduction2}. 
It follows for the same reason that $\F''_p$ is free over $\O_p$ for all $p\in \D_m$. 
We conclude that $(s_1,\dots, s_{r-u})$ is basic in $\F$ at all $p\in \D_m$.

Assumption (3) implies that $(s_1,\dots, s_{r-u})$ is $(n+r-v-\depth \O_p)$-basic in $\F$ at all $p\in \D_n-\D_m$.
Combining with the above, we see that 
\[
(s_1,\dots, s_{r-u})\text{ is }\min(r-u, n+(r-v)-\depth \O_p)\text{-basic in }\F\quad \forall p\in \D_n.
\]
Since $r = u+\rank \ker f \ge v$, by \Cref{BasicDepth} there are $(r-v)$ twisted sections $s_{v-u+1}',\dots, s_{r-u}'$ of $\F$ such that $(s_{v-u+1}',\dots, s_{r-u}')$ is basic in $\F$ at all $p\in \D_n$.
Moreover, these twisted sections can be chosen to correspond to a direct summand $\L$ of $\ker f$ involving the largest $(r-v)$ degrees.
The snake lemma applied to the commutative diagram of exact sequences
\[
\begin{tikzcd}
0 \arrow[r]& \L \arrow[r]\arrow[d]& \F \arrow[r,"g"]\arrow[d,equal]& \F' \arrow[r]\arrow[d,"h"]& 0\\
0 \arrow[r]& \ker f \arrow[r] & \F \arrow[r,"f"]& \F''\arrow[r] & 0
\end{tikzcd}
\] 
implies that $\ker h$ is a direct summand of $\ker f$ involving the smallest $(v-u)$ degrees.
The same argument in the proof of \Cref{Reduction2} shows that $\F'$ is $(T_n)$ of well-defined rank $v$.
\end{proof}

\section{Geometric Applications}\label{Application}

In this section, we explore some geometric consequences of \Cref{Factor}.

%

\subsection*{Points in $\P^2_k$}
\begin{definition}
A zero dimensional subscheme $Z$ of $\P^2_k$ satisfies the \emph{Cayley-Bacharach property} $(C\!B_l)$ if the following holds:
for any colength one subscheme $Z'\subset Z$, any effective divisor in $|\O(l)|$ that contains $Z'$ must also contain $Z$.
\end{definition}

\begin{proposition}
With notations as above, if $Z$ satisfies $(C\!B_l)$ then $Z$ satisfies $(C\!B_{l-1})$.
\end{proposition}
\begin{proof}
Let $Z'\subset Z$ be a subscheme of colength one. 
Suppose $F\in H^0(\O(l-1))$ gives a divisor $V(F)$ that contains $Z'$. 
Let $L \in H^0(\O(1))$ cut out a hyperplane avoiding $Z$. 
It follows that the hypersurface $V(F\cdot L)$ contains $Z'$ and thus contains $Z$.
By the choice of $L$, we conclude that $Z$ is contained in $V(F)$. 
\end{proof}

\begin{definition}
Let $Z$ be a zero dimensional subscheme of $\P^2_k$. 
The \emph{Cayley-Bacharach index} of $Z$, denoted by $C\!B(Z)$, is the largest integer $l$ such that $Z$ satisfies $(C\!B_l)$. 
\end{definition}

The following theorem of Griffith-Harris \cite{GH} characterizes when a set of reduced points satisfies the Cayley-Bacharach property.
The non-reduced case is proven by Catanese \cite{Catanese}.

\begin{theorem*}[Griffith-Harris-Catanese]
Let $Z$ be a zero dimensional subscheme of $\P^2_k$. 
There is an extension of the form 
\[
0 \to \O(-l) \to \E \to \I_Z \to 0
\]
where $\E$ is a rank 2 bundle if and only if $Z$ is a local complete intersection satisfying $(C\!B_{l-3})$. 
\end{theorem*}

\begin{theorem}
Let $Z$ be a zero dimensional local complete intersection in $\P^2_k$. 
Let $S$ be the polynomial ring of $\P^2_k$ and let 
\begin{align*}
0 \to \bigoplus_{i = 1}^t S(-a_i) \to \bigoplus_{i = 1}^{t+1} S(-b_i) \to I_Z \to 0 \tag{$*$}\label{Res}
\end{align*}
 be a minimal graded free $S$-resolution of the homogeneous ideal $I_Z$. 
Suppose $a_1\le \dots \le a_t$, then $a_1-3\le C\!B(Z)\le a_t-3$. 
\end{theorem}
\begin{proof}
We sheafify (\ref{Res}) to obtain a reduction $f:\bigoplus_{i = 1}^{t+1}\O(-b_i) \to \I_Z$. 
Since $Z$ is a local complete intersection, we have $\mu((\I_Z)_p) \le 2$ for all $p\in \P^2_k$. 
Since $\I_Z$ is a non-zero ideal sheaf, it is $(T_1)$ and has rank 1.
Taking $m = u = 1$ and $n =v = 2$, we deduce from \Cref{Factor} that there is an extension of the form 
\[
0\to \O(-a_1) \to \E \to \I_Z \to 0
\]
where $\E$ is $(T_2)$ of rank $2$. 
Since $X$ is regular of dimension $2$, $\E$ is a rank 2 bundle. 
It follows from the theorem above that $C\!B(Z) \ge a_1-3$.

Conversely, suppose $Z$ satisfies $(C\!B_{l-3})$, then there is a rank 2 bundle $\E$ and an extension
\[
0 \to \O(-l) \to \E \to \I_Z \to 0.
\]
Since $\Ext^1(\O(a),\O(b)) = 0$ for all $a,b\in \mathbb{Z}$, the surjection $f$ lifts to a map $\bigoplus_{i = 1}^{t+1}\O(-b_i)\to \E$ and we obtain an exact sequence
\[
0 \to \bigoplus_{i = 1}^t \O(-a_i) \xrightarrow{\varphi} \O(-l)\oplus \bigoplus_{i = 1}^{t+1} \O(-b_i) \to \E \to 0.
\]
If $l>a_t$ then the map $\bigoplus_{i = 1}^t \O(-a_i) \to \O(-l)$ would be zero, and $\varphi$ would drop rank on $Z$.
This is a contradiction to the fact that $\varphi$ drops rank nowhere as $\E$ is a bundle. 
\end{proof}

In particular, the Cayley-Bacharach index of an $(a,b)$-complete intersection in $\P^2_k$ is exactly $a+b-3$. 
Let $Z$ be a zero dimensional local complete intersection in $\P^2_k$ of degree $8$. 
Then $Z$ lies on at least two linearly independent cubic curves $C_1$ and $C_2$.
If $Z$ does not lie on any conic, then $C_1$ and $C_2$ cut out a complete intersection $K$. 
It follows that $Z$ is residual to one point $p$ in $K$.
Since $K$ satisfies $(C\!B_3)$, it follows that every cubic containing $Z$ contains the 
residual point $p$ as well.
This is the classical Cayley-Bacharach theorem.

\subsection*{Curves in $\P^3_k$}

Recall the Lazarsfeld-Rao procedure of producing a curve.

\begin{theorem*}[Lazarsfeld-Rao {\cite[Lemma 1.1]{Rao}}]
Let $C$ be a pure codimension two curve in $\P^3_k$.
The index of specialty of $C$ is defined to be $e(C) := \sup\{l\mid H^1(\O_C(l))\ne 0\}$.
There is an extension of the form
\[
0 \to \bigoplus_{i = 1}^t \O(-a_i)\to \E \to \I_C \to 0,
\]
where $a_1\le \dots \le a_t \le e(C)+4$, and $\E$ is a rank $(t+1)$ bundle with $H^2_*(\E) = 0$.
\end{theorem*}

We show that the above procedure factors through a Hartshorne-Serre correspondence (see \cite{Hartshorne}) if and only if the curve is a generic complete intersection.

\begin{theorem}
With notations as above, the reduction $f:\E \to \I_C$ factors through a reduction $g:\E \to \F$ where $\F$ is a rank 2 reflexive sheaf if and only if $C$ is a generic complete intersection.
In this case, we may choose $g$ such that the factor map $h:\F \to \I_C$ has $\ker h \cong \O(-a_1)$, i.e. $C$ is the vanishing scheme of a section of $\F$ in degree $a_1$.
\end{theorem}
\begin{proof}
The pure codimension two curve $C$ is a generic complete intersection if and only if $\mu((\I_C)_p) \le 2$ for all $p\in \C_2 = \D_2$. 
If there is a reduction $h:\F \to \I_C$ where $\F$ is rank 2 reflexive, then $\mu((\I_C)_p) \le \mu(\F_p) = 2$ for all $p\in \D_2$. 
The converse and the last statement follows from \Cref{Factor} with $m = u = 1$ and $n = v = 2$.
\end{proof}

\subsection*{Codimension two local complete intersections}

\begin{theorem}
Let $X$ be a smooth projective variety of dimension $d$ over an infinite field. 
If $V$ is a Cohen-Macaulay subscheme of pure codimension two such that $\mu((\I_V)_p) \le \dim \O_p$ for all $p\in X-\C_1$, then there is a rank $d$ bundle $\E$ and an extension of the form
\[
0 \to \bigoplus_{i = 1}^{d-1} \O(-a_i) \to \E \to \I_V \to 0.
\]
In other words, every such $V$ is the degeneracy locus of $d-1$ sections of a rank $d$ bundle.
\end{theorem}
\begin{proof}
Since $V$ is of pure codimension two and $X$ is smooth, there are forms $F \in H^0(\I_V(a))$ and $G\in H^0(\I_V(b))$ for some $a,b > 0$ such that $V(F,G)$ is a complete intersection $K$ in $X$. 
Let $Z$ be the subscheme defined by the property that $\I_{Z/K} = \sHom(\O_V,\O_K)$, which is also Cohen-Macaulay of pure codimension two in $X$.
See \cite{PS} for a proof of this fact. 

Let $\bigoplus_{i = 1}^u \O(-a_i) \xrightarrow{\varphi} \I_Z$ be a surjection, we may replace $\varphi$ by another surjection
\[
\varphi':\O(-a)\oplus\O(-b) \oplus \bigoplus_{i = 1}^u \O(-a_i)  \xrightarrow{[F,G,\varphi]} \I_Z.
\]
Since $(\I_Z)_p$ is isomorphic to $\O_p$ for $p\not\in Z$ and is perfect of height two in $\O_p$ for $p\in Z$, it follows that $(\ker \varphi')_p$ is free for all $p\in X$. 
We conclude that $\ker \varphi'$ is a bundle on $X$. 
Let 
\[
\K_\bullet:\quad 0 \to \O(-a-b) \to \O(-a)\oplus \O(-b) \to \O
\]
be the Koszul complex of $\O_K$ and let 
\[
\F_\bullet: \quad 0\to \ker \varphi' \to \O(-a)\oplus \O(-b)\oplus \bigoplus_{i = 1}^u \O(-a_i) \to \O
\]
be the locally-free resolution of $\O_Z$ given by $\varphi'$.
The surjection $\O_K \to \O_Z$ induces a map of complexes 
\[
\begin{tikzcd}
\K_\bullet  \arrow[d,"\alpha"] &  & 0 \arrow[r] & \O(-a-b) \arrow[d] \arrow[r]& \O(-a)\oplus \O(-b)\arrow[r] \arrow[d] & \O \arrow[d]\\
\F_\bullet &  & 0 \arrow[r] & \ker \varphi' \arrow[r] & \O(-a)\oplus \O(-b) \oplus \bigoplus_{i = 1}^u \O(-a_i) \arrow[r]& \O.
\end{tikzcd}
\]
The middle vertical arrow is a split injection by our choice of $\varphi'$. 
The left vertical arrow exists by the left-exactness of taking global sections.
The mapping cone of $\alpha$ splits off the trivial summands $\O(-a)\oplus\O(-b) \xrightarrow{[1,1]} \O(-a)\oplus\O(-b)$ and $\O \xrightarrow{1} \O$ to yield a complex
\[
\G_\bullet: \quad 0 \to \O(-a-b) \to \ker \varphi' \to \bigoplus_{i = 1}^u \O(-a_i) \to 0.
\]
By the standard property of the mapping cone, the complex $\G_\bullet$ is a locally-free resolution of the quotient sheaf $\I_Z/\I_K$. 
However, we note that
\[
\I_Z/\I_K  = \sHom(\O_V,\O_K) \cong \sHom(\O_V,\omega_K)\otimes \omega_X^{-1}(-a-b) \cong \omega_V \otimes \omega_X^{-1}(-a-b).
\]
In other words, the complex $\G_\bullet \otimes \omega_X(a+b)$ is a locally-free resolution of $\omega_V$.
Since $V$ is Cohen-Macaulay of pure codimension two, we can apply the functor $\sHom(-,\omega_X)$ to the complex $\G_\bullet \otimes \omega_X(a+b)$ and obtain a locally-free resolution of $\O_V$
\[
0 \to \bigoplus_{i = 1}^u \O(a_i-a-b) \to (\ker \varphi')^*(-a-b) \to \O.
\]
This gives us an extension of the form 
\[
0 \to \bigoplus_{i = 1}^u \O(a_i-a-b) \to \E \to \I_V \to 0
\]
where $\E$ is the bundle $(\ker \varphi')^*(-a-b)$.
If $\rank \E < d$, then we may add trivial complexes of the form $0 \to \O \xrightarrow{1} \O \to 0 \to 0$. 
Suppose $\rank \E > d$.
Since $\mu((\I_V)_p)\le \dim \O_p = \depth \O_p$ for all $p\in X-\C_1$ and $\I_V$ is $(T_1)$ of rank 1, we may apply \Cref{Factor} with $m = u = 1$ and $n = v = d$ to the reduction $f:\E \to \I_V$, and obtain a factor reduction $h:\F'\to \I_V$ where $\F'$ is of rank $d$. 
\end{proof}

The condition on $V$ is satisfied when $V$ is a local complete intersection of pure codimension two in $X$ (e.g. $V$ is smooth of codimension two).
It is known to experts of liaison theory that such $V$ is the degeneracy locus of $(r-1)$ sections of a rank $r$ bundle.
Our contribution here is an upper bound on $r$ by the dimension of $X$.
This upper bound is sharp: not all smooth codimension two subschemes of a smooth projective variety $X$ of dimension $d$ are the degeneracy loci of $d-2$ sections of rank $d-1$ bundles. 
For example, every smooth curve in $\P^3_k$ is the degeneracy locus of 2 sections of a rank $3$ bundle, but a smooth curve $C$ in $\P^3_k$ that is the vanishing locus of a section of a rank 2 bundle must be subcanonical, i.e. $\omega_C\cong \O_C(l)$ for some integer $l$. 
Clearly not every smooth curve is embedded this way.

\section*{Appendix: Proofs of \Cref{BasicCodim} and \Cref{BasicDepth}}

We largely follow the treatment in \cite{EE} and \cite{Bruns} in the affine case.
We slightly improve the proofs to ``catch'' multiple basic elements at once, in order to avoid another induction that applies the existence theorems multiple times.
In the words of Artin, we are ``putting a general position argument (\Cref{FiniteBasic}) in general position".

\subsection*{Codimension Version}

Note that for any sheaf $\F$, there is a presentation $\E_1 \to \E_0 \to \F \to 0$ where $\E_1$ and $\E_0$ are locally-free of finite rank. 

\begin{definition}
Let $\E_1\xrightarrow{\varphi} \E_0$ be a map of locally-free sheaves of finite rank on $X$. 
The $i$-th minor ideal sheaf $\I_i(\varphi)$ is defined as the image ideal of the map $\wedge^i \E_1 \otimes \wedge^i \E_0^* \to \O$ corresponding to the $i$-th exterior map $\wedge^i\E_1 \to \wedge^i \E_0$. 

Let $\F$ be a sheaf and let $\E_1\xrightarrow{\varphi} \E_0 \to \F\to 0$ be a presentation of $\F$ by locally free sheaves $\E_1$ and $\E_0$ of finite rank.
We define the \emph{$i$-th Fitting ideal} $\Fitt_i(\F)$ of $\F$ to be $\I_{n-i}(\varphi)$, where $n = \rank \E_0$. 
Let $Z_i(\F)$ be the subscheme corresponding to $\Fitt_i(\F)$. 
\end{definition}

\begin{proposition}
With notations as above, the ideal sheaf $\Fitt_i(\F)$ is well-defined and does not depend on the presentation chosen for any $i$.
The subscheme $Z_i(\F)$ contains exactly points $p\in X$ where $\mu(\F_p)>i$. 
\end{proposition}
\begin{proof}
See \cite[\S 20]{CommAlg} for basic facts on Fitting ideals.
\end{proof}

\begin{lemma}\label{Finite1}
Let $\C$ be a set of points in $X$ and let $\F'$ be a subsheaf of a sheaf $\F$.
If $\F'$ is $w$-basic in $\F$ at all points in the following set
\[
\{ p \in X\mid \exists q\in \C \text{ such that $p$ is a generalization of $q$}\},
\]
then $\F'$ is $w$-basic at all but finitely many points in $\C$.
\end{lemma}
\begin{proof}
We claim that if $\F'$ is not $w$-basic at $p\in \C$, then $p$ is the generic point of a component of $Z_i(\F)$ for some $i$.  
Since there are only finitely many ideals $\Fitt_i(\F/\F')$, and each $Z_i(\F)$ has only finitely many components by the noetherian property, the conclusion follows.
Let $p\in \C$ and $\mu((\F/\F')_p)=s$.
It follows that $p \in Z_{s-1}(\F/\F')$ and $p \not\in Z_s(\F/\F')$.
Suppose $p$ is not the generic point of a component, then there exists a point $q \in Z_{s-1}(\F/\F')$ that is a proper generalization of $p$.
By assumption $\mu((\F/\F')_q)\le \mu(\F_q)-w$.
Since $q$ is a generalization of $p$, it follows that $q \not\in Z_s(\F/\F')$. 
We conclude that $\mu((\F/\F')_q) = \mu((\F/\F')_p) = s$.
On the other hand $\mu(\F_q)\le \mu(\F_p)$ since $q$ is a generalization of $p$. 
It follows that $\F'$ is $w$-basic at $p$.
\end{proof}

\begin{lemma}\label{Finite2}
Let $s_1,\dots, s_u$ be twisted sections of a sheaf $\F$.
Suppose for some $t\ge 1$
\[
(s_1,\dots, s_u)\text{ is } \min(u,m+t-\dim \O_p)\text{-basic in }\F\quad \forall p\in \C_m.
\] 
Then $(s_1,\dots,s_u)$ is $\min(u,m+t+1-\dim \O_p)$-basic in $\F$ at all but finitely many $p\in \C_m$.
\end{lemma}
\begin{proof}
Fix an integer $i$ where $0\le i\le m$, and set $\C := \C_i-\C_{i-1}$ to be the set of points in $X$ of codimension exactly $i$.
If $q$ is a generalization of a point in $\C$, then $(s_1,\dots, s_u)$ is $\min(u,m+t-\dim \O_q)$-basic in $\F$ at $q$ and therefore is $\min(u,m+t+1-i)$-basic in $\F$ at $q$. 
By \Cref{Finite1}, there are only finitely many points $p$ in $\C$ where $(s_1,\dots, s_u)$ is not $\min(u,m+t+1-\dim \O_p)$-basic in $\F$ at $p$.
\end{proof}

\begin{proof}[Proof of \Cref{BasicCodim}]
If $u\le t$, then $u\le m+t-\dim \O_p$ for all $p\in \C_m$ and the statement is trivial. 

Suppose $u >t$. 
By \Cref{Finite2}, there are only finitely many points $p_1,\dots, p_v$ in $\C_m$ where $(s_1,\dots, s_u)$ is not $\min(u,m+t+1-\dim \O_p)$-basic. 
We apply \Cref{FiniteBasic} to find forms $r_j \in H^0(\O(a_j-a_1))$ for $2\le j\le u$ such that $(s_2+r_2s_1,\dots, s_u+r_us_1)$ is $\min(u-1,m+t-\dim \O_p)$-basic at $p_1,\dots, p_v$. 
At all points $p$ in $\C_m-\{p_1,\dots, p_v\}$, the subsheaf $(s_2+r_2s_1,\dots, s_u+r_us_1)$ is also $\min(u-1,m+t-\dim\O_p)$-basic in $\F$ since $(s_1,\dots, s_u)$ is $\min(u, m+t+1-\dim \O_p)$-basic.
In summary, we found sections $s_2',\dots, s_u'$ of $\F$ in degrees $a_2,\dots, a_u$ of desired format such that $(s_2',\dots, s_u')$ is $\min(u-1,m+t-\dim\O_p)$-basic in $\F$ at all $p\in \C_m$.

We obtain the desired $t$ sections by applying the above procedure $(u-t)$-times.
\end{proof}

\subsection*{Depth Version}

We recall the definition of the grade of an ideal and extend this definition to ideal sheaves.
Let $R$ be a noetherian ring let $I$ be an ideal of $R$. 
The \emph{grade} of $I$, denoted by $\grade(I)$, is defined to be $\inf \{ i\mid \Ext^i_R(R/I,R) \ne  0\}$. 
Note that if $I = R$ then $\grade(I) = \infty$ by convention, otherwise $\grade(I)$ is a natural number.

\begin{proposition}\label[proposition]{Grade}
With notations as above, if $I\ne R$ then every maximal regular sequence of $R$ in $I$ has length $\grade(I)$. 
Moreover, we have $\grade(I) = \inf \{ \depth R_P\mid P \in V(I)\}$. 
\end{proposition}
\begin{proof}
See \cite[\S 1.2]{BH} for basic facts on the grade of an ideal.
\end{proof}

The latter property allows us to define the grade of an ideal sheaf.
\begin{definition}
If $X$ is a noetherian scheme and $\I$ is an ideal sheaf, then we define 
\[
\grade(\I) := \inf\{\depth \O_p\mid p\in V(\I)\}.
\]
\end{definition}

\begin{lemma}\label[lemma]{FiniteDepth1}
If $\I$ is an ideal sheaf of $X$, then there are only finitely many $p\in V(\I)$ such that $\depth \O_p = \grade(\I)$. 
\end{lemma}
\begin{proof}
Suppose $V(\I)$ is not empty and set $\depth \I = d$.
We may cover $X$ by finitely many affine open subschemes $U_i = \Spec R_i$ and prove that there are only finitely many points $p$ in each $V(\I)\cap U_i$ such that $\depth \O_p = d$. 
Let $I_i$ be the ideal of $R_i$ corresponding to the restriction of $\I$ on $U_i$. 
If $I_i = R_i$ then there is no point $p$ in $V(\I) \cap U_i$, and the statement is vacuously true. 
We may assume $I_i \subsetneq R_i$. 
In this case
\[
d = \inf \{\depth \O_p\mid p\in V(\I)\} \le \inf \{ \depth \O_p \mid p \in V(\I)\cap U_i\} = \grade(I_i).
\]
By \Cref{Grade}, there is a regular sequence $x_1,\dots, x_d$ of $R_i$ in $I_i$. 
If $P$ is a prime ideal of $R_i$ containing $I_i$, then $x_1,\dots, x_d$ is a regular sequence contained in $P$. 
If $\depth (R_i)_P = d$, then $\depth (R_i/(x_1,\dots, x_d))_P = \depth (R_i)_P/(x_1,\dots, x_d)_P = 0$.
It follows that $P$ is an associated prime of $R_i/(x_1,\dots, x_d)$. 
Since there are only finitely many associated primes of $R_i/(x_1,\dots, x_d)$, the conclusion follows.
\end{proof}

\begin{lemma}[cf. \Cref{Finite2}]\label[lemma]{FiniteDepth2}
Let $\F$ be a sheaf where 
\begin{align*}
\mu(\F,q) \le \mu(\F,p)\quad \forall p,q\in X \text{ such that }\depth \O_q < \depth \O_p \le m\tag{$\dagger$}.
\end{align*}
Let $s_1,\dots, s_u$ be twisted sections of $\F$. 
Suppose for some $t\ge 1$, 
\[
(s_1,\dots, s_u)\text{ is }\min(u, m+t-\depth \O_p)\text{-basic in }\F\quad \forall p\in \D_m.
\] 
Then $(s_1,\dots, s_u)$ is $\min(u,m+t+1-\depth \O_p)$-basic in $\F$ at all but finitely many $p\in \D_m$.
\end{lemma}
\begin{proof}
We claim that for fixed values of $a := \depth \O_p$ and $b := \mu(\F,p)-\min(u,m+t+2-a)$, there are only finitely many $p\in \D_m$ where $\depth \O_p = a$ and
\[
\mu(\F/(s_1,\dots, s_u),p) > b.
\]
Since there are only finitely many possibilities of $a$ and $b$, the lemma follows from this claim.

If $q\in X$ is any point where $\depth \O_q < \depth \O_p = a \le m$, then 
\begin{align*}
\mu(\F/(s_1,\dots, s_u),q) & \le \mu(\F,q) - \min(u,m+t+1-\depth \O_q)\\
& \le \mu(\F,p) -\min(u,m+t+1-\depth \O_p) \\
&= b
\end{align*}
by our assumption on $\F$.
We conclude that $q\not\in Z_b(\F/(s_1,\dots, s_u))$ and therefore 
\[
\grade \Fitt_b(\F/(s_1,\dots, s_u)) \ge a.
\] 
By \Cref{FiniteDepth1}, there are only finitely many $p\in Z_b(\F/(s_1,\dots,s_u))$ where $\depth \O_p = a$. 
\end{proof}

\begin{proof}[Proof of \Cref{BasicDepth}]
The proof is the same as that of \Cref{BasicCodim}, with $\D_m$ in place of $\C_m$ and with an application of \Cref{FiniteDepth2} in place of \Cref{Finite2}.
\end{proof}


\begin{thebibliography}{9}


\bibitem{Bass}
Bass, Hyman.
\emph{$K$-theory and stable algebra}.
Inst. Hautes \'Etudes Sci. Publ. Math. No. 22 (1964), 5–60.

\bibitem{Bruns}
Bruns, Winfried.
\emph{"Jede`` endliche freie Auflösung ist freie Auflösung eines von drei Elementen erzeugten Ideals}. (German)
J. Algebra 39 (1976), no. 2, 429–439.

\bibitem{BH}
Bruns, Winfried; Herzog, J\"urgen.
\emph{Cohen-Macaulay rings}. 
Cambridge Studies in Advanced Mathematics, 39. 
Cambridge University Press, Cambridge, 1993. 
{\rm xii}+403 pp. ISBN: 0-521-41068-1

\bibitem{Catanese}
Catanese, Fabrizio.
\emph{Footnotes to a theorem of I. Reider}. Algebraic geometry (L'Aquila, 1988), 67–74,
Lecture Notes in Math., 1417, Springer, Berlin, 1990.

\bibitem{CommAlg}
Eisenbud, David.
\emph{Commutative algebra. With a view toward algebraic geometry.}
Graduate Texts in Mathematics, 150. Springer-Verlag, New York, 1995. xvi+785 pp. 

\bibitem{EE1}
Eisenbud, David; Evans, E. Graham, Jr.
\emph{Basic elements: Theorems from algebraic K-theory}.
Bull. Amer. Math. Soc. 78 (1972), 546–549.

\bibitem{EE}
Eisenbud, David; Evans, E. Graham, Jr.
\emph{Generating modules efficiently: theorems from algebraic K-theory}.
J. Algebra 27 (1973), 278–305.

\bibitem{Forster}
Forster, Otto.
\emph{\"Uber die Anzahl der Erzeugenden eines Ideals in einem Noetherschen Ring}. (German)
Math. Z. 84 (1964), 80–87.

\bibitem{GH}
Griffiths, Phillip; Harris, Joseph.
\emph{Residues and zero-cycles on algebraic varieties}.
Ann. of Math. (2) 108 (1978), no. 3, 461–505.


\bibitem{Hartshorne}
Hartshorne, Robin.
\emph{Stable reflexive sheaves}.
Math. Ann. 254 (1980), no. 2, 121–176.

\bibitem{Rao}
Lazarsfeld, Robert; Rao, Prabhakar.
\emph{Linkage of general curves of large degree}. Algebraic geometry - open problems (Ravello, 1982), 267–289,
Lecture Notes in Math., 997, Springer, Berlin, 1983.

\bibitem{Mumford}
Mumford, David.
\emph{Lectures on curves on an algebraic surface}.
With a section by G. M. Bergman. Annals of Mathematics Studies, No. 59 Princeton University Press, Princeton, N.J. 1966 xi+200 pp.


\bibitem{PS}
Peskine, Christian; Szpiro, Lucien.
\emph{Liaison des vari\'{e}t\'{e}s alg\'{e}briques. I}. (French)
Invent. Math. 26 (1974), 271–302.

\bibitem{Swan}
Swan, Richard G.
\emph{The number of generators of a module}.
Math. Z. 102 (1967), 318–322.

\bibitem{Zhang}
Zhang, Mengyuan.
\emph{Biliaison of sheaves}. Preprint
\href{https://arxiv.org/abs/2004.02280}{arXiv:2004.02280}


\end{thebibliography}
\end{document}